\documentclass[12pt,twoside,a4paper,notitlepage]{article}

\usepackage{graphicx}
\usepackage{mathtools}
\usepackage{enumerate}
\usepackage{amssymb}
\usepackage{mathptmx}     
\usepackage{csquotes}

\usepackage{latexsym}
\newtheorem{theorem}{Theorem}

\newtheorem{definition}{Definition}
\newtheorem{proof}{Proof}
\newtheorem{example}{Example}
\newtheorem{remark}{Remark}
\newtheorem{acknowledgements}{Acknowledgements}

\begin{document}

\title{Bolzano's Infinite Quantities}

\author{Kate\v{r}ina Trlifajov\'{a}}

\maketitle

\begin{abstract}
In his \emph{Foundations of a General Theory of Manifolds}, Georg Cantor praised Bernard Bolzano as a clear defender of actual infinity who had the courage to work with infinite numbers. At the same time, he sharply criticized the way Bolzano dealt with them. Cantor's concept was based on the existence of a \emph{one-to-one correspondence}, while Bolzano insisted on Euclid's Axiom of \emph{the whole being greater than a part}. Cantor's set theory has eventually prevailed, and became a formal basis of contemporary mathematics, while Bolzano's approach is generally considered a step in the wrong direction. 

In the present paper, we demonstrate that a fragment of Bolzano's theory of infinite quantities retaining the \emph{part-whole principle} can be extended to a consistent mathematical structure. It can be interpreted in several possible ways. We obtain either a linearly ordered ring of finite and infinitely great quantities, or a a partially ordered ring containing infinitely small, finite and infinitely great quantities. These structures can be used as a basis of the infinitesimal calculus similarly as in Non-standard Analysis, whether in its full version employing ultrafilters due to Abraham Robinson, or in the recent \enquote{cheap version} avoiding ultrafilters due to Terrence Tao. 

\end{abstract}

\section{Introduction}
\subsection{Existence of an actual infinity}

There are two major questions that anyone dealing with actual infinity has to address:
\begin{enumerate}
\item Is there any actual infinity?
\item If so, are there multiple infinities? How can we compare them? 
\end{enumerate}

An answer to the first question has been sought by many eminent thinkers, philosophers and mathematicians, starting from Aristotle, who formulated the question, to Cantor (\v{S}ebest\'{i}k 1992). Aristotle rejected actual infinity, and traditional scholasticism responded similarly in the famous thesis of \emph{Infinitum actu non datur.} It was not until the end of the 19th century that Georg Cantor introduced actual infinity in mathematics in the form of infinite sets. He gave a theological justification for their existence, placing them in the mind of God (Dauben 1979), (Halett 1986). A characteristic excerpt from Cantor's letter to Jeiler from 1895, quoted in (Halett 1986, p. 21), says \enquote{in particular, there are transfinite cardinal numbers and transfinite ordinal types which, just as much as the finite numbers and forms, possess a definite mathematical uniformity, discoverable by men. All these particular modes of the transfinite have existed from eternity as ideas in the Divine intellect.} Once Cantor's set theory was accepted and axiomatized, there has no longer been need to answer the questions above. The answer has ultimately been given in the form of an axiom: in the Zermelo-Fraenkel set theory, the axiom is called the \emph{Axiom of Infinity:}\footnote{For instance in (Fraenkel et al. 1973, p. 47).} $$(\exists z)((\emptyset \in z) \wedge (\forall x)(x \in z \Rightarrow x \cup \{x\} \in z)).$$

There is an important consequence of the Axiom of Infinity and of the \emph{Axiom of the Power-set}: \footnote{In symbols $(\forall a)(\exists y)(\forall x)(x \in y \Leftrightarrow x \subseteq a)$, (Fraenkel et al. 1973, p. 35). Of course the other axioms are important too, but these two are substantial for the construction of infinite ordinal and cardinal numbers.} the existence of infinitely many infinite ordinal and cardinal numbers.

\subsection{Measuring the actual infinity}

It is well-known that if we accept the existence of actual infinity and try to compare infinite collections, we have to choose one of the following two, mutually exclusive, principles:  
\begin{enumerate}[1.]
\item The \emph{part-whole principle,} also known as the 5th Euclid Axiom: The whole is greater than its part. 
\item The \emph{principle of one-to-one correspondence,} or Hume's Principle: Two sets have the same size if and only if there is a one-to-one correspondence between them. 

\end{enumerate}

The two principles cannot be valid simultaneously for infinite collections. This fact is called the \emph{Paradox of Reflexivity.} It was known already to Galileo Galilei who employed natural numbers and their squares to demonstrate it. Based on this paradox, Galileo claimed that the attributes of \enquote{equal}, \enquote{greater} and \enquote{smaller} could not be applied to infinite, but only to finite quantities. Similarly, Leibniz refused to compare infinite collections in size, although he did not reject actual infinity \emph{per se}.\footnote{Both Bolzano and Cantor quote the famous passage from Leibniz's letter to \enquote{I am so in favor of the actual infinite that instead of admitting that Nature abhors it, as is commonly said, I hold that Nature makes frequent use of it everywhere, in order to show more effectively the perfections of its Author} (Cantor 1976, p. 78),  (Russ 2004, p. 593.)}

Cantor based his set theory on the second principle. The one-to-one correspondence is decisive: from its existence, one can derive the \enquote{size} of sets -- Cantor calls it the \emph{cardinality.} The part-whole principle then fails dramatically, for all infinite sets. 

Bolzano built the foundations of his theory of infinity on the first principle. Here is a pertinent comment by Guillerrmina Waldegg, who investigated the approach of students to infinity and to comparison of infinite sets:
\begin{quote}
Bolzano's criterion, based on the part-whole correspondence is more \enquote{intuitive} because it is nearer to concrete (finite) experience (in fact it corresponds to the mainstream of the historical moment, and to common sense in all times). On the other hand, a one-to-one relationship is less visible. Cantor's solution could only be reached by means of a total detachment from meaning and intuition; his solution is therefore more problematic and difficult to understand, but it was well accepted once it proved its potential in the mathematical field. (Waldegg 2005, p. 574).
\end{quote}

Prior to Bolzano, there were several attempts to find alternative answers, by Grosseteste, Maignan, Rodrigo da Arriaga, and others, cf.\ the paper \emph{Measuring the Size of Infinite Collections of Natural Numbers: Was Cantor's Set Theory Inevitable?} (Mancosu 2009). Bolzano, however, went the farthest of all. The main part of Bolzano's theory of infinity is contained in his last treatise, the \emph{Paradoxes of the Infinite}, which he wrote in 1848, the year of his death. It was published only posthumously in 1851 (Bolzano 1851), (Russ 2004). We proceed primarily from this work, though we take Bolzano's other books into consideration as well. The paragraph numbers given below refer to the paragraphs in the \emph{Paradoxes of the Infinite}.

\subsection{Opinions on Bolzano's theory of infinity}

Bolzano's theory is usually perceived from the point of view of Cantor's set theory. The latter is viewed as the only correct theory. Bolzano's conception thus appears to be either incorrect, or an imperfect predecessor of Cantor's. 

Georg Cantor himself expressed this view, having read Bolzano's \emph{Paradoxes of the Infinite}. He mentioned them in his 1883 work \emph{Foundations of a General Theory of Manifolds}. While praising Bolzano on the one hand as an intrepid supporter of actual infinity, he nonetheless considered Bolzano's concept to be insufficient and erroneous.
\begin{quote}
The genuine-infinite as we encounter it, for 
example in the case of well-defined point sets, \dots has found its most decisive defender in a philosopher and mathematician of our century with a most acute mind, Bernhard Bolzano, who has developed his views relevant to the subject especially in the beautiful and substantial essay, \emph{Paradoxes of the Infinite.} \dots
Bolzano is perhaps the only one who assigns the genuine-infinite numbers their
rightful place. \dots However, the actual way in which he deals with
them, without being able to advance any kind of real definition of them, is something about which I am not at all in agreement with him, and I regard for example Sections 29-33 of that book as
\emph{unfounded and erroneous}. The author lacks both the general concept
of cardinality and the precise concept of number-of-elements for a real conceptual grasp of determinate infinite numbers. (Cantor 1976, p. 78).
\end{quote}

This view of Cantor's persists until today. For example, in his book \emph{Labyrinth of Thought} Jos\'{e} Ferreiros says of Bolzano
\begin{quote}
He proposed to base mathematics on notions similar to the set. He made a clear defense of actual infinity and he proposed precise notions for treating infinite sets. In this way he even came quite close to such a central notion of set theory as cardinality. But after having been close to the \emph{right} point of view, he departed from it in quite a \emph{strange} direction. (Ferreiros 1999, p. 75). 
\end{quote}
The \enquote{right point of view} is the idea that cardinality is the only meaningful way to compare abstract sets with respect to the multiplicity of their elements, as Ferreiros explains in a note. 

Experts on Bolzano are of a similar opinion. Paul Rusnock, who translated Bolzano's \emph{Theory of Science} into English, is very doubtful about Bolzano's conception\footnote{Rusnock argues using Bolzano's own definition of the set, according to which it does not matter how its elements are ordered. At the same time, however, if we want to declare two sets to be equal in terms of the size of their elements, it is also necessary that they have the same determining ground [Bestimmungsgrunde].  Rusnock considers this a mistake and rejects Bolzano's conception. See (Rusnock 2000). But Bolzano's definitions are not contradictory. Even Cantor's concept of sets could be rejected in a similar way. The existence of a one-to-one correspondence between two sets entails their ordering.} 

\begin{quote}
Bolzano did not see this point, and thus 
stopped on the boundary of Cantor's set theory of cardinal numbers. \dots Bolzano falls short of a full and satisfactory treatment of the theory of the infinite. (Rusnock 2000, pp .194, 196). 
\end{quote} 

Jan \v{S}ebest\'{i}k believes that Bolzano's interpretation is disputable as it wavers between two principles for comparing infinite sets. 

\begin{quote}
L'\'{e}vidence intuitive de l'axiom [du tout et de le part], d\'{e}riv\'{e}e de la consideration des ensembles finis, interdit \`a Bolzano, comme \`a ces pr\'{e}d\'{e}cesseurs, d'\'{e}tablir un crit\`ere d'\'equivalence par l'interm\'ediaire d'une bijection. Ce conflit profond, qui \'eclate presque dans chaque section des \emph{Paradoxes de l'infini}, l'emp\^eche de construire une arithm\'etique de l'infini qui soit coh\'erente. (\v{S}ebest\'{i}k 1992, p. 464). \end{quote}

Jan Berg, an expert on Bolzano and publisher of his works, found a passage in one of Bolzano's late letters to Robert Zimmermann which Berg interprets, in an effort to save Bolzano's reputation, as Bolzano changing his opinion and accepting the existence of one-to-one correspondence as a sufficient criterion for the equality of sets.
\begin{quote} Hence, it seems that in the end Bolzano confined the doctrine that the whole is greater than its part to the finite case and accepted isomorphism as a sufficient condition for the identity of cardinalities of infinite sets. (Mancosu 2009, p. 625). \end{quote}
Both Rusnock (Rusnock (2000), pp. 194 - 196.) and Mancosu examine this passage and consider Berg's interpretation to be groundless. Mancosu says, tongue-in-cheek  
\begin{quote}Thus, Bolzano saved his mathematical soul in extremis and joined the rank of the blessed Cantorians by repudiating his previous sins. ... I must observe that the literature of infinity is replete with such `Whig' history. (Mancosu 2009, p. 626). 
\end{quote} 

The situation is similar that of the historiography of infinitesimals. After the scorching critique of Berkeley and the triumph of the \enquote{great triumvirate} - Weierstrass, Cantor, Dedekind -  consequently after the introduction of the Archimedean continuum as the sole possible and analysis based on limits, they were mostly considered, at best, erroneous or, at worst, altogether inconsistent. Nevertheless the group of mathematicians around Mikhail Katz demonstrates in a number of papers that the view of infinitesimals through the lens of the Weierstrassian foundation is misguided. The title of the first paper is typical \emph{Is mathematical history written by the victors?}(Bair et al. 2013) The framework of Robinson's non-standard analysis is in many cases more helpful in understanding the work of Gregory, Fermat, Leibniz, Euler and Cauchy.  (Blaszczyk et al. 2017), (Bair et al. 2017), (Bascelli et al. 2017).

\section{Bolzano's theory of infinite multitudes}

Bolzano writes that he wants to contribute to the question of what is the infinite in general. He deals with infinite pluralities because 

\begin{quote}\dots if it should be proved that, strictly speaking, there is nothing other than pluralities to which the concept of infinity can be applied in its true meaning, i.e. if it should be proved that infinity is really only a property of a plurality or that everything which we have defined as \emph{infinite} is only called so because, and in so far as, we discover a property in it which can be regarded as an infinite plurality. Now it seems to me that is really the case. (\S 10.)
\end{quote} 

\subsection{Collections, multitudes, pluralities, series, quantities}

At the beginning of \emph{Paradoxes of the Infinite}, Bolzano very precisely and thoroughly explains the basic notions in \S\S 3 - 9. Nearly the same explanations can be found in the \emph{Theory of Science} in \S\S 82 - 86. The basic notion is 
\enquote{\emph{a collection} of certain things or a whole consisting of certain parts,} which are connected by the conjunction \enquote{and}.

Bolzano had dealt with the concept of \emph{collections} before, and there has been considerable professional discourse about his idea. (Simons 1998), (Rusnock 2000), (Behboud 1998), (Krickel 1995).  This primarily concerns the question of whether we should interpret the objects contained in a collection as elements or as parts thereof. Peter Simons shows that Bolzano's collections cannot be equated even with Cantor's sets, nor can they serve as a basis for mereology, the study of parts and wholes, as Franz Krickel has proposed, but that they constitute a separate and specific notion
\begin{quote} Bolzano's theory of collections is best interpreted as a distinct and distinctive theory of collections. (Simons 1998, p. 87).\end{quote} 

In his previous works, Bolzano tried to define the concept of the collection as generally as possible. Problems and ambiguities arise in particular around collections containing concrete parts, as Paul Rusnock points out. These are not substantive for our purposes, as we will be dealing only with certain abstract types of collections. We will proceed from Bolzano's quoted definition from the \emph{ Paradoxes of the Infinite.} What is important is that the parts contained in the collection are connected by a conjunction \enquote{and} and constitute a whole.

\emph{Collections [Inbegriff]} can be given as a certain list, for instance \enquote{the sun, the earth and the moon, or the rose and the concept of a rose}. (\S 3). Collections can be also denoted by some idea $A$, which we call the collection of all $A$. (\S 14). Parts of the collections are different. \footnote{This again poses a difficult question of how to interpret Bolzano's statements about variety. But let us stick to this simple interpretation.} The parts of the collections can, but need not be, arranged in some way.

\emph{A multitude [Menge]}\footnote{This is difficult to translate and there is no clear-cut solution. We adhere to the translation proposed and substantiated by Peter Simons and used by Steve Russ: \emph{Inbegriff = collection, Menge = multitude, Vielheit = plurality, Grosse = quantity.} We should point out that various English translations and articles use different terminology.} 
is a collection that is conceived such that the arrangement of its parts is unimportant. (\S 4). 
Bolzano uses the word \emph{Menge}.\footnote{Georg Cantor also uses the word in his articles from 1880 onwards instead of the original \emph{Mannichfaltigkeit} or \emph{Inbegriff}.} The English translation \emph{set} can be misleading, because it refers to the Cantor's conception of a set. The \enquote{set} is on a different \enquote{higher} level than its elements. \footnote{This is also reflected in the method of notation in which we use curly brackets. The symbol $\{a,b,c\}$ denotes a multitude containing elements $a,b, c$, whereas Bolzano would have written it as $a + b + c$. 
Likewise, if we want to indicate that a group of three items comprises a multitude, we draw a circle around it.} Bolzano's multitude, [Menge], is on the \enquote{same level} as the things it contains, because it is given by the usage of the connective \enquote{and}. Therefore, it is meaningful to speak about \emph{parts} of a collection. These parts, however, are not subsets. Subsets are also on a \enquote{higher} level than their elements.

Bolzano introduces multitudes for the same reason as Cantor introduces sets. 
In order to capture an infinite plurality altogether, it must be collected in a single collection, a single, currently infinite, whole.

A special type of a multitude is \emph{a plurality [Vielheit]}. Its parts are considered units of a certain kind $A$. (\S 4).

\emph{A series [Reihe]} is also a collection of things, but unlike a multitude, it is ordered. Bolzano denotes things which the series contains as $A, B, C, \dots, L, M, N \dots $. We can always determine the following term $N$ of the series from the preceding term $M$ by the special \emph{rule of formation} (\S 7).

If we have a series in which the first term is a unit of the $A$ kind and every following term is derived from the previous one as a sum of its predecessor and a new unity of the kind $A$, then all of the terms of this series are \emph{pluralities}, or, more straightforwardly, \emph{numbers.} (\S 8).
Some pluralities have so many terms, that they exhaust all the units and they have no last term. 

Now it is more evident why Bolzano defines a collection as a whole consisting of certain parts. A series is an example of such a whole. Its sum is a quantity of the same kind, it is a sum of all its parts.

A plurality is called \emph{infinite} if it is greater than every finite plurality, i.e. if every finite multitude represents only a part of it. (\S 9).

\emph{A finite multitude} is not explicitly defined here. We find the explanation in \S 22. The numbers of its parts are equal to some natural number $n$. 
\begin{quote} 
If we know that $A$ is finite, we designate some arbitrary thing in the multitude $A$ by 1, some other multitude by 2, etc. \dots, we must sometimes arrive at a thing in $A$ after which nothing remains which is still undesignated. Now let this last thing just spoken of in $A$ get the number $n$ for its designation, then the number of things in $A = n$. 
\end{quote}

Moreover, Bolzano speaks about \emph{quantities}. This too is one of Bolzano's complicated terms. (Simons 2004). Here we will stick to Bolzano's definition as given in \emph{Paradoxes of the Infinite}. A \emph{quantity, [Grosse]}, belongs to a kind of thing of which every two can have no other relation to one another than either being equal to one another, or one of them presenting itself as a sum that includes a part equal to the other one (\S 6). We see that pluralities are also quantities.

\emph{An infinite quantity} is greater than every finite number of the unit taken. An \emph{infinitely small quantity} is so small that every finite multiple of it is smaller than the unit. (\S 10).
\footnote{The seventh section of \emph{Pure Theory of Numbers} is devoted to Infinite Quantity Concepts. Bolzano distinguishes infinitely small, measurable, and infinitely large Infinite Quantity Concepts. More on this in the next section.}

A quantity $A$ is \emph{infinitely greater} than $B$, if $A$ is infinitely greater than every finite number of $B$; we will denote that here as $A >> B$.\footnote{Bolzano doesn't define this notion explicitly, but in \S 10 he speaks about \enquote{infinitely smaller and infinitely greater \emph{quantities of higher order}, which all proceed from the same concept}. In \S 33 he proves that $S$ is infinitely greater than $P$ because $S - n \cdot P > 0$ for any finite number $n$, see Theorem \ref{S>P}.} 

Bolzano defended the existence of actual infinity decades before Cantor. Cantor underpinned his belief in the existence of infinite cardinal and ordinal numbers with theology.
(Dauben 1979, pp. 228 - 232.) Bolzano also proves the existence of \emph{the infinite multitude of truths in themselves}.\footnote{This proof has two variations of which one is included in the \emph{Theory of Science} (Bolzano 1837/2014) and the other in the \emph{Paradoxes of the Infinite} in \S 13.} He proceeds from a proposition which can be proved by contradiction: \enquote{There are truths.} Bolzano asserts that if we have any true proposition $A$, then the proposition \enquote{A is true} is also true and it is different. We can construct infinitely many truths by this way, in fact by induction. Bolzano finally uses the theological argument
\begin{quote}
Thus we must attribute to Him [God] a power of knowledge that is true omniscience, that therefore comprehends an infinite multitude of truths. (\S 11.)
\end{quote}
Nevertheless, Bolzano tries to get by without a theological argument. It could be made possible by the special mode of the existence of truths and propositions in themselves. (\v{S}ebest\'{i}k 1992, pp. 446-447).

Bolzano then refers to the similarity of the multitude of these propositions and the multitude of natural numbers and infers that this multitude too must be infinite. He notes that there is no last number among natural numbers, nor can there be, as the very notion conceals a contradiction. (\S 15). Similarly, a multitude of all fractions (of rational numbers) and irrational numbers is also infinite. The multitude of all quantities including infinitely small and infinitely great quantities must be infinite too. (\S 16).

In all Bolzano's demonstrations of infinite multitudes, the emphasis is not on their \emph{existence} but on their \emph{infiniteness.} The existence of collections (multitudes) follows from their definition. But it's necessary to prove they are infinite.

\subsection{Part-whole principle and one-to-one correspondence}\label{correspondence}

Bolzano claims unequivocally to the \emph{part-whole principle.} Not all infinite multitudes are to be regarded as equal to one another with respect to their plurality. If one multitude is contained in another as a mere part of it then the former is smaller than the latter.  

The example, \enquote{to whom must it not be clear,} is  the straight line with two points $a$ and $b$. The length of the straight line from the point $a$  continuing without limit in the direction to the point $b$ is infinite. It may be called greater than the straight line from the point $b$ continuing without limit in the same direction, by the piece $ab$. And the straight line continuing without limit on both sides may be called greater by a quantity which is itself infinite. (\S 19).

Simultaneously Bolzano is also aware of the existence of the \emph{one-to-one correspondence} between some infinite multitudes. He is aware that sometimes one multitude is a part of the other one simultaneously. The following quotation begins a little ambiguously. 
\begin{quote}Let us now turn to the consideration of a highly remarkable peculiarity which \emph{can occur}, indeed actually \emph{always occurs}, in the relationship between two multitudes \emph{if they are both infinite}, but which previously has been overlooked to the detriment of knowledge \dots (\S 20.)\end{quote} 
Bolzano describes very carefully the one-to-one correspondence.\footnote{\enquote{It is possible to combine each thing belonging to one multitude with a thing of the other multitude, thus creating pairs, with the result being that no single thing remains without connection to a pair, and no single thing appears in two or more pairs (\S 20.)}} He presents two examples of two multitudes such that, \emph{on the one hand}, there is a one-to-one correspondence between them and, \emph{on the other hand}, one of these multitudes contains the other within itself as a mere part (\S 20). 

The first example are the multitudes of quantities\footnote{\label{IQC} When Bolzano writes about quantities in the interval $(0, 5)$ he apparently has on mind his \emph{measurable numbers}. Their construction is described in his \emph{Infinite Quantity Concepts}, see the Section \ref{measurable}.} 
in two open intervals $(0,5)$ and $(0,12)$. \enquote{Certainly the latter multitude is greater since the former is indisputably only a part of it.} 

The one-to-one correspondence is defined by the equation 
$$5y = 12x.$$

The second example is similar: the infinite multitudes of points of line segments $ab$ and $ac$, where $b$ is an inner point of line $ac$. \enquote{The multitude of points which lie in $ac$ exceeds that of the points in $ab$, because in $ac$ as well as all the points of $ab$ there also lie that of $bc$ which do not occur in $ab$.} 

Let $x$ be some point in $ab$ then it corresponds to the point $y$ in $ac$ if the following proportion holds 
$$ab : ac = ax : ay.$$

Bolzano emphatically warns against the assertion that one-to-one correspondence of two multitudes allows for the conclusion of the equal plurality of their parts if they are infinite. 
\begin{quote} An equality of these multiplicities can only be concluded if some other reason is added, such as that both multitudes have exactly the same \emph{determining ground [Bestimmungsr\H{u}nde]}, e.g. they have exactly the same \emph{way of being formed [Entstehungsweise]}. (\S 21). \end{quote} 

Bolzano doesn't explain here what the \emph{determining ground} means exactly. He speaks in other texts about \emph{elements of determination, [bestimmende St\H{u}cke]}, though it is not entirely clear there either. To determine an object means to describe all representations that the objects falls under. The determination is complete if the representation of an object is unique. Thus, a point, the area and the diameter completely determine a circle, because all of its properties can be uniquely determined. (\v{S}ebest\'{i}k 1992, p. 460).   

Bolzano shows that this is not the case in the examples mentioned. (\S 23). 
In the first Example, the distance of two quantities in $(0,5)$, say $3$ and $4$, is less than the distance of their corresponding quantities in $(0,12)$ which are $7\frac{1}{5}$ and $9\frac{3}{5}$. It is similar in the second example. The distance $ax$ is different (smaller) than the distance $ay.$ Obviously, Bolzano considers the distance to be a determining element. In this case we can speak about isometry rather than a mere one-to-one correspondence. An isometry is a distance-preserving transformation between two metric spaces.

On the other side, Bolzano presents the positive example of multitudes which have the same plurality of elements. (\S 29).
He denotes the multitude of quantities between two numbers $a$ and $b$ as $Mult(b-a)$. This multitude is infinite and depends only on the distance of the boundary quantities and therefore must be equal whenever this distance is equal. Therefore, we have innumerable equations of the forms 
$$Mult(8 - 7) = Mult(13-12). $$
$$Mult(b - a) : Mult (d - c) = (b-a) : (d-c)$$

One-to-one correspondence is a necessary condition but not a sufficient one for the equality of two multitudes. The necessary condition in the cited examples is the \emph{isometry}. In the case of geometrical objects, Bolzano says:

\begin{quote}
Every spatial extension that is not only similar to another but also geometrically equal (i.e. coincides with it in all characteristics that are conceptually representable through comparison with a given distance) must also have an equal multitude of points. (\S 49. 2.) 
\end{quote}

At this point, Bolzano fundamentally breaks with Cantor. For Bolzano it is still true that \emph{the whole is greater than its part} and the existence of the \emph{one-to-one correspondence} does not attest to the equality of multitudes. We will see how Bolzano deals with this further on. 

\subsection{Calculation of the infinite}

The sums of series represent infinite quantities. Is it possible to determine their relationship? Is it possible to compare them or to count them? Bolzano answers that 
\begin{quote}a calculation of the infinite done correctly doesn't aim at calculation of that which is determinable through no number, namely not a calculation of infinite pluralities in itself, but only a determination of the relationship of one infinity to another which is feasible, in certain cases, as we shall show in several examples. (\S 28). 
\end{quote}

In fact, our knowledge about Bolzano's calculation of the infinite comes from examples. Nevertheless, they are clear enough that we can derive the rules. Bolzano makes an important assumption that all series have \emph{one and the same multitude of terms}. (\S 33). He does not ask what the plurality of the terms is, but assumes it is \emph{always the same.} The sum of the series represents a whole, a uniquely given quantity. If we add positive terms, then the quantity increases; if we take away positive terms, the quantity decreases.

It is similar when we consider the sum of the series
$$1/2 + 1/4 + 1/8 + \dots + \text{in inf.}$$ 
We are not interested in the plurality of the terms, but we are interested in its sum, which is equal to $1$.\footnote{It's equal in the meaning of equality introduced by Bolzano in the \emph{Infinite Quantity Concepts}, i.e. if their difference is infinitely small, see \ref{measurable}.} 

\begin{example} \label{rady}Bolzano introduces several examples of infinite series (\S 29, \S 33).
\begin{enumerate}
\item $N = 1 + 1 + 1 + 1 + \dots + \text{in inf.} $
\item $P = 1 + 2 + 3 + 4 + \dots + \text{in inf.} $
\item $S = 1 + 4 + 9 + 16 + \dots + \text{in inf.} $
\item $N_n = \dots 1 + 1 + 1 + \dots + \text{in inf.} $

The last series is similar to the first one, but we only add from the $(n+1)$-st term. The first $n$ terms are omitted. We designate this as $N_n$.
\end{enumerate}
\end{example}

Bolzano shows several examples and explains how to count the infinite series. 

\begin{theorem}\label{n}
If we designate the number of natural numbers from the $n+1$ by $N_n$ then we obtain the equation. 
$$n = N - N_n.$$
\end{theorem}

Bolzano doesn't prove this assertion, he calls it the \enquote{certain and quite unobjectionable equation from which we see how two infinite quantities $N$ and $N_n$ have a completely definite finite difference} (\S 29). His persuasion is based on the implicit assumption that the multitude of the terms of the series is always one and the same. Consequently, every term of the series is important. From the proof of the following theorem we see again the same principle. 

\begin{theorem}\label{e}
If $0 < e < 1$ then 
$$1 + e + e^2 + e^3 + \dots + \text{in inf.} = \frac{1}{1-e}$$
\end{theorem}

\begin{proof} The proof begins like this
\begin{itemize}
\item[] $S = 1 + e + e^{n-1} + e^n + e^{n+1} + \dots + \text{in inf.}$
\item[] $S = \ \quad \frac{1 - e^n}{1 - e} \quad + \quad e^n + e^{n+1} + \dots + \text{in inf.} \quad = \quad \frac{1 - e^n}{1 - e} + P^1 $, $P_1$ designates a positive quantity which is dependent on $n$ an $e$.
\item[] $S = \ \quad \frac{1 - e^n}{1 - e} \quad + \quad e^n [1 + e + \dots + \text{in inf.}]$ 
\end{itemize}
\begin{quote}
The quantity $[1 + e + \dots + \text{in inf.}]$ is not to be regarded as identical to $S$, since the set of terms is not the same, rather here it is indisputably $n$ terms less than in $S$. 
\footnote{The proof of this theorem is stated in the note in \S 18. Here is the end of Bolzano's proof
\begin{itemize} 
\item[] We designate $[1 + e + \cdots + \text{in inf.} ] = S - P^2$, in which we assume that $P^2$ designates a quantity which is positive and dependent on $n$. We obtain
\item[]$S = \frac{1 - e^n}{1 - e} + e^n(S - P^2) = \frac{1 - e^n}{1 - e} + e^n S - e^n P^2 $
\item[]$S (1 - e^n) = \frac{1 - e^n}{1 - e} - e^n P^2$
\item[]$S = \frac{1}{1 - e} - \frac{e^n}{1 - e^n} P^2.$ Finally we have to prove that $\frac{e^n}{1 - e^n} P^2$ can be brought down below any $1/N$.

\item[] $S = \frac{1 - e^n}{1 - e} + P^1 = \frac{1}{1 - e} - \frac{e^n}{1 - e^n} P^2$, thus $ \frac{e^n}{1 - e^n} P^2 + P^1 = \frac{e^n}{1 - e}.$ 
\item[] If we take $n$ arbitrarily great, thereby the value $\frac{e^n}{1 - e}$ is brought down below under every arbitrary quantity $1/N$, and the quantities $ \frac{e^n}{1 - e^n} P^2$ and $P^1$ must themselves fall below any arbitrary value. Consequently 
\item[] $S = \frac{1}{1 - e}.$ 
\end{itemize}}
\end{quote}
\end{proof} 
The following theorem and its proof demonstrates Bolzano's way of adding infinite series and the sufficient condition of their order. He employs Galileo's example of the series of natural numbers and the series of their squares (second powers), so $P$ and $S$ from the example above. The question isn't which series has more terms - the multitude of terms is the same - but which sum of the series is greater. 
$$P = 1 + 2 + 3 + 4 + \dots + \text{in inf.} $$
$$S = 1 + 4 + 9 + 16 + \dots + \text{in inf.} $$

\begin{theorem}\label{S>P} Let $P, S$ be the sums of infinite series from the example \ref{rady}. Then
\begin{itemize}
\item $S$ is greater than $P$, $S > P$. 
\item $S$ is infinitely greater than $P$, $S >> P$. (\S 33).
\end{itemize} 
\end{theorem}

\begin{proof} Bolzano's reasoning is: \enquote{The multitude of terms in both series is certainly the same. By raising every single term of the series $P$ to the square into the series $P$, we alter merely the nature (magnitude) of these terms, not their plurality.} 
\begin{itemize}
\item Every term of $S$, except the first one, is greater than the corresponding term of $P$, consequently $$S > P.$$ 
\item We shall prove that $S$ is infinitely greater than $P$. If we subtract successively the series $P$ from $S$ then we obtain the remainders

$S - P = 0 + 2 + 6 + 12 + \dots + \text{in inf.} $

$S - 2 P = -1 + 0 + 3 + 8 + \dots + \text{in inf.} $

$S - m P = (1-m) + \dots + (n^2-mn) + \dots + \text{in inf.} $

In these series, only a finite multitude of terms, namely $m-1$, is negative, and the $m$th term is 0, but all successive terms are positive and grow indefinitely. The sum of every series is positive. Thus $$S >> P.$$ 
\end{itemize}
\end{proof}
Consequently, there is infinitely many infinite quantities. Sometimes we can determine their sum, difference, or order. 

\subsection{More infinite quantities}

Some infinite sequences represent finite quantities, although they are composed of an infinite multitude of fractions, for instance irrational quantities or geometrical series.

\begin{example} Let $e < 1$. 
\begin{enumerate}
\item 
$$ \frac{14}{10} + \frac{1}{100} + \frac{4}{1000} + \frac{2}{10000} + \dots + \text{in inf.} = \sqrt 2$$
\item $$ a + ae + ae^2 + ae^3 + \dots + \text{in inf.} = \frac{a}{1-e}$$
\end{enumerate}
\end{example}

Bolzano presents examples of infinite quantities of higher orders and infinite quantities which have any rational or irrational order. So he extends the notion of the infinite quantities. 
\begin{example}

Let $\alpha$ be any quantity then the following series are also quantities.  $\alpha \cdot N$ is also a quantity. (\S 29).
\begin{enumerate}
\item $$\alpha \cdot N = \alpha + \alpha + \alpha + \dots + \text{in inf.}$$ 
\item $$N^2 = N + N + N + N + \dots + \text{in inf.} $$
\item $$N^3 = N^2 + N^2 + N^2 + N^2 + \dots + \text{in inf.} $$

\end{enumerate}
\end{example}

\begin{theorem}\label{pomer}
\begin{enumerate}
\item Let $\alpha, \beta$ be two quantities then $$(\alpha \cdot N) : (\beta \cdot N) = \alpha : \beta.$$ 
\item $N^2$ is infinitely greater than $N$, $$N^2 >> N.$$
\item $N^3$ is infinitely greater than $N^2$, $$N^3 >> N^2.$$
\end{enumerate}
\end{theorem} 

In \emph{ Paradoxes of the Infinite} Bolzano investigates in particular series that have infinite sums. He investigated infinitely small quantities earlier in \emph{Infinite Quantity Concepts}; see the following section. Nevertheless, he mentions them here as well. 

\begin{example}
If $N$ is any infinitely great quantity, then the infinitely small quantity is represented by 
$$1/N.$$
\end{example}
Obviously, Bolzano tries to create a general theory of quantities.
\begin{quote}
We have infinitely many infinitely great and infinitely small quantities, of which they have every arbitrary ratio one to another. In particular, they can be infinitely greater or infinitely smaller. Therefore, there also exist infinite orders among infinitely large and infinitely small quantities, and it is indeed possible to find frequently correct equations between quantities of this kind. (\S 30)
\end{quote}

\subsection{Conditions for infinite series}\label{conditions} 
Not all series can be considered quantities. The first condition was presented in the definition of the series. A \emph{rule of formation} determines the following term from the previous one.

The further condition is presented in $\S 32$. Bolzano says
\begin{quote}
I may draw attention to the fact, which is not itself incomprehensible, that there may be \emph{quantity expressions} which designate \emph{no actual quantity}.\dots In particular, a \emph{series}, if we want to consider it only as a quantity \dots has such a nature that it undergoes no change in value when we make a change in the order of its terms. With quantities, it must be that 
$$(A + B) + C = A + (B + C) = (A + C) + B.$$ 
\end{quote}
To put it in contemporary terms, the sum must be associative and commutative.

\begin{example}
Bolzano's counterexample is a series
$$a - a + a - a + a - a + \dots \text{in inf.}. $$ 
If we change parentheses or the arrangement of the terms we obtain different quantities. 
\begin{itemize}
\item[] $\ \ 0 = (a - a) + (a - a) + (a - a) + \dots \text{in inf.}$
\item[] $\ \ a = a + (- a + a) + (- a + a) + \dots \text {in inf.}$
\item[] $-a = -a + (a - a) + (a - a) + (a - a) + \dots \text{in inf.}$
\end{itemize}
Another counterexample is $$1 - 2 + 4 - 8 + \dots \text{in inf.}$$ 
\end{example}

The meaning of the criterion of the associativity and commutativity of series is not entirely clear. The question is, for instance, whether the following sequence would be considered a quantity or not. 
$$1 + 2 + 1 + 2 + 1 + 2 + \dots \text{in inf.}$$ 
Probably not, because then the series $2 + 1 + 2 + 1 + 2 + 1 + \dots \text{in inf.}$ would also be a quantity, and their difference would be the problematic series $1 - 1 + 1 - 1 + 1 - 1 + \dots \text{in inf.}$ So what does the requirement of associativity and commutativity mean? Bolzano's aim was to create a general theory of quantities. He needed sums, differences, and maybe also products and ratios of quantities to be quantities as well.

\subsection{Measurable numbers}\label{measurable}

Before we present an interpretation of this fragment of Bolzano's theory of infinity, we should recall Bolzano's theory of measurable numbers. It is found in the manuscript \emph{Infinite Quantity Concepts}, the Seventh Section of \emph{Pure Theory of Numbers} from the early 1830's. It was not published until 1962, when it was discovered by Karel Rychl\'{i}k and prepared for print. There was a great deal of dispute over it; perhaps we can conclude that it can be considered a theory of real numbers \emph{cum grano salis}.\footnote{Russ, S., Trlifajov\'{a}, K. 2016.}

The construction is similar to the construction of infinite pluralities, but Bolzano employs not only infinite sums but infinitely many arithmetic operations used on integers. There is also an implicit assumption that the infinite multitude of operations is always the same. The basic notion is an infinite quantity expression, which is a generalization of the notion of rational numbers. While rational numbers can be considered \emph{quantity expressions} in which only a finite number of arithmetic operations (addition, subtraction, multiplication and division) on integers is required, in \emph{infinite quantity expressions} an infinite multitude of arithmetic operations is required. 
\begin{example} Bolzano introduces several examples.
\begin{enumerate}
\item $1 + 2 + 3 + 4 + \dots \text{in inf.}$
\item ${1 \over 2} - {1 \over 4} + {1 \over 8} - {1 \over 16} + \dots \text{in inf.}$
\item $(1 - {1 \over 2})(1 - {1 \over 4})(1 - {1 \over 8})(1 - {1 \over 16}) \dots \text{in inf.}$
\item $a + \frac{b}{1+1+1+ \dots \text{in inf.}}$ where $a, b$ is a pair of whole numbers. 
\end{enumerate}
\end{example}

Bolzano designates a \emph{measurable number} as an infinite quantity expression $S$ for which the following condition holds true. If for every positive integer $q$ we determine the integer $p$ such that 
$$ S = \frac{p}{q} + P_1 = \frac{p+1}{q} - P_2.\footnote{This is the main defect of Bolzano's theory of measurable numbers. Everything would be entirely consistent if this condition was slightly different.
$$S = \frac{p-1}{q} + P_1 = \frac{p+1}{q} - P_2.$$ There are several possible explanations but that is not the subject of this paper.} $$
where $P_1$ and $P_2$ denote a pair of strictly positive quantity expressions (the former possibly being zero).

Infinite number expression is \emph{infinitely small} if its absolute value is less than $1 \over q$ for any natural number $q$, and it is \emph{infinitely great} if it is greater than any $q$.

Bolzano designates that two measurable numbers $S, P$ are \emph{equal} or \emph{equivalent} , if \emph{their difference is infinitely small}.\footnote{It is a similar relation as the Leibniz equality \emph{up to} denoted by the symbol \enquote{$\sqcap$} (Bascelli et al. 2016).} In contemporary terms, he makes a factorization by this relation. Nevertheless, Bolzano still calls them measurable numbers, which is slightly confusing. He introduces the \emph{ordering} of measurable numbers: $P > S$, if their difference is positive and not infinitely small. Bolzano shows that measurable numbers defined in this way have all usual properties of real numbers: it is an \emph{Archimedean, dense, linearly ordered field}, where the \emph{Bolzano-Cauchy theorem}\footnote{This theorem in particular was essential. As early as 1817 in \emph{Purely Analytical Proof} Bolzano defined a Bolzano-Cauchy sequence and wrote that such a sequence has a limit. He proved its uniqueness, but nevertheless he could not prove its existence. There was no theory of real numbers.} and the \emph{supremum theorem} is valid.\footnote{Of course Bolzano does not use these terms. 
We write them as commonly known designations. 
Bolzano describes the properties explicitly, for instance the \emph{Archimedean property}: 
\begin{quote} If $A$ and $B$ are a pair of measurable and finite numbers, both of which we also consider positive or absolute, then there is always some multiple of one, which is \emph{greater} than other, and some aliquot part of the same one which is \emph{smaller} than other one. (Bolzano 1976, \S 74) or (Russ 2004, p.399). \end{quote} 
Bolzano formulated and proved all substantial properties of real numbers.} 

\subsection{Interpretation of measurable numbers}

If we wish to demonstrate the consistency of measurable numbers, we interpret infinite quantity expressions as sequences of partial results, so they correspond to sequences of \emph{rational numbers}. It is easy to prove that measurable numbers correspond to Bolzano-Cauchy sequences, infinitely small quantities to sequences converging to 0 and infinitely great quantities to divergent sequences. Now, we can proceed in one of two ways. We can either follow the standard Cantor's construction of real numbers or we can compare it with the non-standard construction of real numbers. 

Cantor defines that two Bolzano-Cauchy (fundamental) sequences of rational numbers are equal if their difference is a sequence converging to 0. Bolzano defines that two measurable numbers are equal if their difference is infinitely small. The meaning is the same. Consequently measurable numbers with equality has the same structure as Cantor's real numbers. It is  formally suitable, but infinitely small numbers are lost in this interpretation.

The Bolzano's construction of measurable numbers is somewhat similar to the non-standard construction of real numbers.\footnote{
There is a nice construction of \emph{real} numbers from \emph{rational} numbers by means of non-standard analysis: we proceed the usual way, but start from sequences of \emph{rational numbers} $\mathbb Q^\mathbb N$. We employ a non-principal ultrafilter $\mathcal U$ on natural numbers $\mathbb N$ and obtain an ultraproduct $$\mathbb Q^* = \mathbb Q^ \mathbb N/\mathcal U.$$ 

The elements of the ultraproduct are called hyperrational numbers. It is a linearly ordered field, non-Archimedean. We define, in the usual way its finite, infinite and infinitely small elements, and the relation of infinite closeness $\doteq$. Let $\mathbb Q_f$ denote the set of finite elements, $\mathbb Q_i$ the set of infinitely small elements. 
$$\mathbb Q_i \subseteq \mathbb Q_f \subseteq \mathbb Q^*.$$ 
By the factorization of this set by the relation of infinite closeness we obtain the structure isomorphic to real numbers. The result is the same as a factorization modulo $Q_i$. 
$$\mathbb Q_f/\doteq \quad  \cong  \quad \mathbb R.$$
Perhaps this is the most direct way of constructing the reals from rationals. (Albeverio, p.14). Nevertheless this structure is inconvenient for the introduction of the differential and integral calculus. The key \emph{transfer principle} doesn't hold between finite hyperrational numbers and real numbers. Functions defined on real numbers cannot be simply extended on hyperrationals.      
} 
Bolzano does not employ ultrafilters, of course. His measurable numbers (without equality) are sufficient in this case. It is a non-Archimedean commutative ring. 
The set of all infinitely small quantities, corresponding to sequences converging to 0, forms its maximal ideal. 
The introduction of equality of measurable numbers means the factorization by the equality relation, i.e. the factorization of the ring by its maximal ideal. The result is the linearly ordered field, isomorphic to the real numbers.

Bolzano's measurable numbers are inconvenient for the introduction of calculus for the same reason as hyperrational numbers, because the transfer principle fails for them. Bolzano was obviously aware of this fact. He returned to his original idea of building the calculus on the basis of \emph{quantities which can become smaller than any given quantities} which is the idea of the later Weierstassian \enquote{$\epsilon - \delta$ analysis}. 
In the \emph{Theory of Functions} Bolzano defines all the key notions of analysis in this manner. For instance, this is the definition of the continuity of a function in a point $x$. 
\begin{quote} Supposing the value $Fx$ is measurable, as well as the value $F(x + \Delta x)$\footnote{\rm For Bolzano, $\Delta x$ is simply the difference between $x$ and any other modified value (Russ 2004, p. 438).},
but the difference $F(x + \Delta x) - Fx$, in its absolute value, becomes and remains smaller than any given fraction $1 \over N$, providing only that $\Delta x$ is taken small enough then I say that the function Fx varies continuously for the value $x$.  (Russ 2004, p. 448). \end{quote}

While Bolzano concentrated on the arithmetization of continuum and the construction of measurable numbers in \emph{Infinite Quantity Concepts}, in \emph{Paradoxes of the Infinite} he investigated mainly the sums of infinite series of measurable numbers.  

While we interpreted Bolzano's infinite quantity expressions as sequences of \emph{rational numbers} we will justifiably interpret Bolzano's infinite sums as sequences of \emph{real numbers}.

\section{Interpretation of Bolzano's infinite quantities}

\subsection{Sequences}

Bolzano's infinite series are interpreted as infinite sequences of partial sums, i.e. as elements of $\mathbb R^\mathbb N$. But we will keep in mind that these sequences represent single, exactly given \emph{quantities}. There is a slight conceptual distinction between treating elements of $\mathbb R^\mathbb N$ as Bolzano's quantities, and treating them as classical sequences. The two viewpoints are formally equivalent, but can lead to a somewhat different way of thinking about such objects. This is similar to the case of non-standard analysis, where we treat elements of ultraproducts as single quantities.

It is easy and unambiguous to pass from series to sequences and back. For every series there is exactly one corresponding sequence and vice versa.

\begin{definition} Let $a_1 + a_2 + a_3 + \dots + \text{in inf.}$ be the Bolzano series of real numbers, $a_i \in \mathbb R.$ It corresponds to the sequence $(s_1, s_2, s_3, \dots ) \in \mathbb R^\mathbb N$ where for all $n \in \mathbb N, s_n = a_1 + \dots + a_n.$
$$a_1 + a_2 + a_3 + \dots + \text{in inf.} \sim (s_1, s_2, s_3, \dots ) = (s_n),$$ 
\end{definition}
\begin{example}
\begin{enumerate}
\item $N = 1 + 1 + 1 + 1 + \dots + \text{in inf.} \sim (1, 2, 3, \dots) = (n)$ 
\item $P = 1 + 2 + 3 + 4 + \dots + \text{in inf.} \sim (1, 3, 6, 10, \dots ) = (\frac {n \cdot (n+1)}{2})$
\item $S = 1 + 4 + 9 + 16 + \dots + \text{in inf.} \sim (1, 5, 16, 32, \dots ) = (\frac{n(n+1)(2n+1)}{6})$
\item $1 + e + e^2 + e^3 + \dots + \text{in inf.} \sim (\frac{1 - e^n}{1 - e})$
\item $n = 1 + 1 + \dots + 1 \sim (1, 2 \dots n, n, n, \dots)$
\item $N_n = \dots 1 + 1 + 1 + + \dots + \text{in inf.} \sim (0, 0, \dots 1, 2, 3, \dots )$
\item $r \cdot N = r + r + r + \dots + \text{in inf.} \sim (r, 2r, 3r, \dots) = (nr) $ where $r \in \mathbb R$.\footnote{We use the symbol $r$ for a real number instead of Bolzano's symbol $\alpha$ for quantities.}
\end{enumerate}
\end{example}

\begin{remark}
Sequences express well the difference between the Bolzano series $N$ and $N_n$.
\end{remark}

\begin{remark}\label{embedding}
There is a natural interpretation of any real number $r \in \mathbb R$ as a constant sequence
$$r \sim (r) = \bf{r}$$
So real numbers $\mathbb R$ can be naturally embedded in sequences of real numbers. $$\mathbb R \longrightarrow \mathbb R^\mathbb N.$$ 
\end{remark}

We saw how Bolzano counted up the sums of two series. He counted up the corresponding terms. There is no example of multiplication of series in the \emph{Paradoxes of the Infinite}, nevertheless, we can suppose that their product should be the extension of the product of a finite amount of terms. 

\begin{definition} Let $(a_n), (b_n)$ be two sequences of real numbers. We define their sum and their product pointwise:
$$(a_n) + (b_n) = (a_n + b_n).$$
$$(a_n) \cdot (b_n) = (a_n \cdot b_n).$$ 

\end{definition}
\begin{remark}
If there are two Bolzano series $$a_1 + a_2 + a_3 + \dots + \text{in inf.} \sim (a_1, a_1+ a_2, a_1+ a_2+a_3, \dots)$$
$$b_1 + b_2 + b_3 + \dots + \text{in inf.} \sim (b_1, b_1+ b_2, b_1+ b_2+b_3, \dots)$$
then their sum and the product correspond to the sum and the product of sequences
$$(a_1 + a_2 + a_3 + \dots + \text{in inf.})+(b_1 + b_2 + b_3 + \dots + \text{in inf.}) \sim (a_1 +b_1 , a_1+a_2+b_1+ b_2 , a_1+ a_2+a_3+b_1+ b_2+b_3, \dots)$$
$$(a_1 + a_2 + a_3 + \dots + \text{in inf.}) \cdot (b_1 + b_2 + b_3 + \dots + \text{in inf.}) \sim (a_1 \cdot b_1 , (a_1+a_2) \cdot(b_1+ b_2), (a_1+ a_2+a_3) \cdot (b_1+ b_2+b_3), \dots)$$
\end{remark}

\begin{remark}
We will interpret the quantity that is designated by $N^2 = N + N + \dots + \text{in inf.}$ as 
$$N^2 = N \cdot N \sim (1, 2, 3, \dots) \cdot (1, 2, 3, \dots) = (n^2).$$ 
\end{remark}

\subsection{Equality and ordering}

Sequences of real numbers with pointwise operations of addition and multiplication form the ring $\mathbb R ^{\mathbb N}$, i.e. the countably infinite product of copies of the field $\mathbb R$. But this ring is not suitable for our purposes directly; we cannot define infinitely small and infinitely great quantities. In order to avoid this problem, we will not deal with all the sequences in $\mathbb R ^{\mathbb N}$ individually, but identify them by an equality relation respecting the arithmetic operations. 


We can interpret the Bolzano's conditions of commutativity and associativity of the series like this: If we change the order of the finite amount of terms of the series, the sum of the series does not change. Therefore, we define that two sequences are \emph{equal} if their terms are the same from a sufficiently large index. The \emph{order} is defined on the same principle. In fact, we use the Fr\'{e}chet filter here that is the filter containing all complements of finite sets.

\begin{definition} Let $(a_n), (b_n)$ be two sequences of real numbers. We define their equality and order
$$(a_n) =_\mathcal F (b_n) \text{ if and only if } (\exists m)(\forall n)(n > m \Rightarrow a_n = b_n).$$
$$(a_n) <_\mathcal F (b_n) \text{ if and only if }(\exists m)(\forall n)(n > m \Rightarrow a_n < b_n).$$

\end{definition}

The following definitions are in accordance with Bolzano's own words in \emph{Infinite Quantitative Concepts}, see \ref{measurable}. \emph{Infinite closeness} of two sequences corresponds to Bolzano's equality (equivalence) of two measurable numbers.
\begin{definition}
Let $(a_n), (b_n)$ be two infinite series. Then $(a_n)$ is \emph{infinitely smaller} than $(b_n)$ or $(b_n)$ is \emph{infinitely greater} than $(a_n)$

$$(a_n) << (b_n) \text{ if and only if }(\forall k)(a_n) <_\mathcal F k \cdot (b_n) $$ 
where $k, n \in \mathbb N$. 
\end{definition}


\begin{definition} Let $x, y$ be elements of $\mathbb R^\text{o}$. 
\begin{enumerate}[(i)]

\item $(a_n)$ is \emph{infinitely small} if and only if  $(a_n) << \bf{1}.$
\item $(a_n)$ is \emph{infinitely great} if and only if  $(a_n) >> \bf{1}.$
\item $(a_n)$ and $(b_n)$ are \emph{infinitely close}, $(a_n) \doteq (b_n)$ if and only if $(a_n - b_n)$ is infinitely small. 
\end{enumerate}
\end{definition}

\begin{theorem} Bolzano's theorems mentioned here \ref{n}, \ref{e}, \ref{S>P}, \ref{pomer} are valid in this interpretation. \end{theorem}
\begin{proof} Let $n \in \mathbb N, r, s \in \mathbb R$. 
\begin{enumerate}
\item $n = N - N_n$ or $n + N_n = N$ , because $$(n, n \dots n, n, n, \dots) + (0, 0, \dots 1, 2, 3, \dots ) =_\mathcal F (1, 2, 3, \dots n, n+1, n+2, \dots) $$
\item $\frac{1}{1 - e} - (\frac{1 - e^n}{1 - e})$ is infinitely small because
$$(\forall q)(\exists n)(\forall m > n) |\frac{1}{1 - e} - \frac{1 - e^n}{1 - e}| = |\frac{e^n}{1 - e}| < \frac{1}{q}. $$
\item $S > P$, because $$n > 1 \Rightarrow \frac{n(n+1)(2n+1)}{6} > \frac{n(n+1)}{2}$$
$S >> P$, because $$(\forall k)(\forall n)(n > \frac{3k}{2}-\frac{1}{2} \Rightarrow \frac{n(n+1)(2n+1)}{6} > k \cdot \frac{n(n+1)}{2})$$ 
\item $(r \cdot N) : (s \cdot N) = r : s$ where $r, s \in \mathbb R$, because
$$(r\cdot n) : (s \cdot n) = r : s$$
$N^2 >> N \sim (n^2) >> (n)$, because $$(\forall k)(\forall n)(n > k \Rightarrow n^2 > k \cdot n).$$ 
$N^3 >> N^2 \sim (n^3) >> (n^2)$, because $$(\forall k)(\forall n)(n > k \Rightarrow n^3 > k \cdot n^2).$$
\end{enumerate}
\end{proof}

\begin{theorem} The structure $\mathbb R^\text{o} = (\mathbb R ^\mathbb N, +, \cdot, =_\mathcal F, <_\mathcal F )$ is a partially ordered non-Archimedean commutative ring. \end{theorem}

\begin{proof}
For all elements $a,b, c \in \mathbb R^\text{o}$ the following properties are clearly valid
\begin{enumerate}[(1)] 
\item $(a + b) + c =_\mathcal F a + (b + c)$ \quad associativity of $+$ 
\item $a + b =_\mathcal F b + a$, \quad commutativity of $+$ 
\item $a + o =_\mathcal F a$, \quad there is a neutral element with respect to $+$, where $o = (0,0,0, \dots)$
\item $(\exists b)(a + b =_\mathcal F 0)$, \quad existence of an inverse element, where $b = -a$.
\item $(a \cdot b) \cdot c =_\mathcal F a \cdot (b \cdot c)$, \quad associativity of $\cdot$ 
\item $a \cdot b =_\mathcal F b \cdot a$, \quad commutativity of $\cdot$ 
\item $(\forall a)(a \cdot i =_\mathcal F a)$, \quad there is a neutral element with respect to $+$, where $\cdot$, where $i = (1,1,1, \dots)$.
\item $a \cdot (b + c) =_\mathcal F (a \cdot b) + (a \cdot c)$, \quad distributivity.
\item $\neg (a <_\mathcal F a)$, \quad irreflexivity
\item $((a <_\mathcal F b \wedge b <_\mathcal F c) \Rightarrow a <_\mathcal F c)$, \quad transitivity. 
\end{enumerate}
\end{proof}
Elements of $\mathbb R^\text{o}$ correspond to Bolzano's series. They represent finite, infinitely great and infinitely small quantities.

But sequences corresponding to alternating series are not expelled, for example 
$$1 - 1 + 1 - 1 + \dots \text{in inf.} \sim (1, 0, 1, 0, \dots)$$
$$- 1 + 1 - 1 + 1 \dots \text{in inf.} \sim (0,1, 0, 1,  \dots)$$

Such oscillating sequences have no inverse elements with respect to multiplications, they are divisors of zero and they are not ordered.  Consequently $\mathbb R^\text{o}$ is neither a field nor an integer domain and it is not linearly ordered.

\subsection{Polynomials}

If we look for a consistent interpretation just for infinite quantities, there is a criterion that would expel oscillating sequences and still it would be a ring, even a linearly ordered integral domain.\footnote{An integral domain is a commutative ring in which there are no non-zero divisors of zero.}

The terms of the series are to be determined by a \emph{rule of formation} that determines the following term from the preceding one. Let's restrict the rule to the use of addition and multiplication of integers. The following term can then be determined by a polynomial function. Consequently, terms of the corresponding sequences are polynomials.

\begin{definition} 
Let $p(n)$ be a polynomial the $k$-th degree. 
$$p(n) = a_0 + a_1 n + \dots a_k n^k,$$
where $a_i \in \mathbb R, i = 1, \dots, k.$ The following sequence is \emph{polynomial}.
$$(p(n)).$$
\end{definition} 

\begin{theorem}
Polynomial sequences with the above defined operations form a non-Archimedean linearly ordered integral domain. 
\end{theorem}

\begin{proof}
 The sum and the product of two polynomial sequences is a polynomial sequence too, because the sum and the product of two polynomials is a polynomial. The product of two polynomial non-zero sequences is non-zero. 

The ordering is linear, because if $p(n) = a_0 + a_1 n + \dots a_k n^k$ and $q(n) = b_0 + a_1 n + \dots b_k n^k$ then $$(p(n)) >_\mathcal F (q(n)) \Leftrightarrow (a_k, a_{k-1}, \dots a_0) > (b_k, b_{k-1}, \dots b_0) \text{ in a lexicographical order.}$$

The constant sequences $(r)$, $r \in \mathbb R$ are polynomial. For example, the sequence $(n) >> (r)$ for any real number $r$.
\end{proof} 

This is a consistent interpretation, but only of \emph{infinitely great} quantities; no polynomial sequences are infinitely small. 

\subsection{Cheap version of non-standard analysis} 

We saw that Bolzano's infinite expressions are not suitable for the foundation of differential calculus although they contain infinitely small numbers. What about infinite quantities introduced in \emph{Paradoxes of the Infinite}?  

Let's return to the above defined structure of sequences of rational numbers $\mathbb R^\text{o}$. 
From now on, we work with the factorization of $\mathbb R ^\mathbb N$ by the Fr\'{e}chet filter $\mathcal F$.\footnote{In fact, it is the same structure as before, only its formal description is a little different. The equality of sequences from $\mathbb R^\mathbb N$ was defined with help of Fr\'{e}chet filter. Now we work with equivalence classes  factorized by the Fr\'{e}chet filter.}  
$$\mathbb R^\text{o} = \mathbb R ^\mathbb N/\mathcal F.$$  

Its elements are equivalence classes, for any $(a_n) \in \mathbb R^\mathbb N$  $$[(a_n)]_\mathcal F \in \mathbb R^0.$$ 

Of course, it is a partial ordered commutative ring too. Real numbers representing Bolzano's measurable numbers (with equality) are naturally embedded in it, see the Remark \ref{embedding} and the Footnote \ref{IQC}. They are linearly ordered and form a field.

Terence Tao described the so-called \enquote{cheap version of non-standard analysis} which is defined on $\mathbb R^0$ (Tao, T. 2012a). It's less powerful than the full version, it's constructive because it doesn't require any sort of axiom of choice and its assertions can be easily rewritten to the standard analysis. 
The \emph{transfer principle} is valid only partially. The structure is countably saturated and, consequently, non-Archimedean.

\begin{definition}
Every function $f$ defined on real numbers $f: \mathbb R \longrightarrow \mathbb R$ can be extended to the function $f^*: \mathbb R^\text{o} \longrightarrow \mathbb R^\text{o}$ such that we define for all $x = [(x_n)]_\mathcal F \in \mathbb R^\text{o}$ 
$$f^*([(x_n)]_\mathcal F) = [(f(x_n))]_\mathcal F.$$
\end{definition}

We define basic notions of mathematical analysis, for instance
\begin{definition} (Tao, 2012a).
\begin{itemize}
\item The function $f$ is \emph{continuous}, if and only for all $x \in \mathbb R$ and for all $y \in \mathbb R^\text{o}$ $$x \doteq y \Rightarrow f(x) \doteq f(y)$$
\item The function $f$ is \emph{uniformly continuous}, if and only for all $x, y \in \mathbb R^\text{o}$ $$x \doteq y \Rightarrow f(x) \doteq f(y)$$
\item The function $f$ has a \emph{derivative} $f'(x)$ in $x \in \mathbb R$, if and only if for all infinitely small $h \in \mathbb R^\text{o}$ 
$$f'(x) \doteq \frac{f(x+h)-f(x)}{h}$$ 
\end{itemize}
\end{definition}
Infinitesimal calculus can be partially reconstructed in $\mathbb R^\text{o}$ as we see from the result of Terence Tao. Nevertheless this structure bears some disadvantages, it is not an integral domain and it contains infinitely many useless elements such as $[(1,0,1,0, \dots)]_\mathcal F$. 

\subsection{A full version of non-standard analysis} 

Factorization of the ring $\mathbb R ^{\mathbb N}$ by the Fr\'{e}chet filter leads to the partially ordered ring $\mathbb R^\text{o}$ which is not an integral domain. 

Should the ring $\mathbb R ^{\mathbb N}_{\mathcal F}$ contain no non-trivial zero-divisors such as $[(1,0,1,0 \dots)]_\mathcal F$, the filter $\mathcal F$ has to satisfy the following condition: for each subset $S$ of $\mathbb N$, either $S \in \mathcal F$ or $\mathbb N \setminus S \in \mathcal F$. 

Indeed, if  $1_S$ and $1_{\mathbb N \setminus S}$ denote the characteristic functions of the sets $S$ and $\mathbb N \setminus S$, respectively, then $1_S . 1_{\mathbb N \setminus S} = 0$, so either $1_S$ or $1_{\mathbb N \setminus S}$ has to be equivalent to the zero sequence. In the former case, $S \in \mathcal F$, in the latter $\mathbb N \setminus A \in \mathcal F$. 

In other words, $\mathbb R^\text{o}$ contains no non-trivial zero-divisors, iff $\mathcal F$ is an ultrafilter on $\mathbb N$. If we add a further natural condition of equivalence for each pair of sequences that differ only at finitely many places, we see that the desired factor-rings $\mathbb R^\text{o}$ correspond to non-principal ultrafilters $\mathcal U$ on $\mathbb N$. Notice that the ultraproduct $$\mathbb R^* = \mathbb R^{\mathbb N} /{\mathcal U},$$ is even a field. Moreover, $\mathbb R$ embeds into $\mathbb R^*$ as the subfield of the equivalence classes of constant sequences. The immense bonus obtained thanks to \L{o}\'{s} Theorem, (Chang, Keisler 1992, p. 211.) is that not only $\mathbb R^*$ is a field, but it shares \emph{all} first order properties with the field $\mathbb R$. And it contains the infinitesimals.

The ultraproduct is a consistent mathematical framework for differential and integral calculus which is based on infinitesimals, i.e. infinitely small quantities, as it was traditionally used until the end of the 19th century, and as it was justified in the non-standard analysis. (Robinson 1966).

Ultrafilters and their wonderful properties are the result of modern logic of the 20th century. We need the axiom of choice to prove their existence. Bolzano, of course, could not have known that.

Terence Tao speaks about \emph{completions} of mathematical structures in (Tao, 2012b, p. 150-152). Every completed structure is much larger, it contains the original one and brings some new benefits. Thus, real numbers represent a \emph{metric} completion of rational numbers. Every Bolzano-Cauchy sequence of rational numbers has a real limit. Complex numbers represent an \emph{algebraic} completion of real numbers. Every algebraic equation has a solution. The non-standard real numbers represent an \emph{elementary} completion of real numbers. Every sequence of statements of first-order logic that are potentially satisfiable are also actually satisfiable here.\footnote{It is also called a countable saturation property.} 
So, we proceed from natural numbers, step-by-step we complete the structures, and finally we obtain the non-standard real numbers as a basis for the infinitesimal calculus. 

\section{Conclusion}

Bolzano's theory of infinite quantities can be extended to a consistent theory in several distinct ways, depending on what we require from the extended theory. 

If we only wish to create a consistent theory of infinite quantities, we can interpret them as polynomial sequences. This results in a linearly ordered ring of finite and infinite quantities. 

In case we want to avoid ultrafilters, we can define equality and order via the Fr\'{e}chet filter. This approach also yields a compact theory of infinitely small and infinitely large quantities generalizing the real numbers. The resulting structure forms a ring, not a field, and it is only partially ordered, but it makes it possible to reconstruct large parts of the differential and integral calculi. 

Certainly, the most effective and elegant approach is via ultrafilters. We obtain a theory of finite, infinitely great and infinitely small quantities. Infinitely great quantities are inverses of the infinitely small ones and vice versa. The principle that the whole is larger than a part holds true. All of the differential and integral calculi can naturally be developed on this basis, because the extended world containing infinitesimal quantities satisfies the same first order properties as the original setting.

Bolzano's theory provides foundations for a more general theory containing both infinitely great and infinitely small quantities. It justifies Pascal's belief in two types of infinity that reflect each other:
\begin{quote}
Ces deux infinis quoique infiniment differents, sont n\'{e}anmoins relatifs l'un \`a l'autre, de telle sorte que la connaissance de l'un m\`ene n\'{e}cessairement \`a la connaissance de l'autre. \footnote{(Pascal 1866, p. 295).} 
\end{quote}
 
Although Cantor's and Bolzano's approaches to comparing infinite sets are based on different principles, they are not incompatible. They merely describe different facets of infinite sets and could coexist within a single comprehensive theory. The concept of \emph{multitude}, based on \emph{par-whole} principle, as well as the concept of \emph{cardinality}, based on \emph{one-to-one correspondence}. However, I consider the Bolzano's framework more suitable, both for the theory of infinite quantities and as a basis of mathematical analysis.

\begin{acknowledgements}
I am grateful to Mikhail Katz, Jan \v{S}ebest\'{i}k and Petr K\accent23 urka for helpful comments on an earlier version of the manuscript. Many thanks to my colleague Jan Star\'{y} for fruitfull discussions about ultrafilters and to my husband Jan Trlifaj for comments on the final version. 
\end{acknowledgements}

\end{document}